 \documentclass[preprint,review,12pt]{elsarticle}

\usepackage{amssymb}
 \usepackage{amsthm,amsmath}
\usepackage{mathrsfs}

 \numberwithin{equation}{section}

\newtheorem{thm}{Theorem}[section]
\newtheorem{lem}[thm]{Lemma}
\newtheorem{prop}[thm]{Proposition}

\newtheorem{defn}[thm]{Definition}
\theoremstyle{definition}

\allowdisplaybreaks

\journal{Journal}

\begin{document}

\begin{frontmatter}



\title{On the uniqueness for the 2D MHD equations without magnetic diffusion}


\author{Renhui Wan}
\ead{rhwanmath@zju.edu.cn, rhwanmath@163.com}
\address{Department of Mathematics, Zhejiang University, Hangzhou 310027, China}
\begin{abstract}
 In this paper, we obtain the uniqueness of the 2D MHD equations, which fills the gap of recent work \cite{1} by Chemin et al.
\end{abstract}


\end{frontmatter}

\section{Introduction}
This paper considers the 2D MHD equations given by
\begin{equation} \label{1.1}
\left\{
\begin{array}{l}
\partial_t u + u\cdot\nabla u + \nabla p-\nu\Delta u = B\cdot\nabla B,  \\
\partial_t B + u\cdot\nabla B -B\cdot\nabla u=0,\\
{\rm div} u=0,\quad {\rm div}B=0, \\
u(x,0) =u_0(x), \quad B(x,0) =B_0(x),
\end{array}
\right.
\end{equation}
here $t\ge0$, $x\in \mathbb{R}^2$,  $u=u(x,t)$ and $B=B(x,t)$ are vector fields representing
the velocity and the magnetic field, respectively, $p=p(x,t)$ denotes the pressure and $\nu$ is a positive viscosity constant.
\vskip .1in
(\ref{1.1}) has been investigated by many mathematicians. In 2014,  by establishing a generalized Kato-Ponce estimate (see \cite{Kato} for the well-known result):
\begin{equation*}
<u\cdot\nabla B\mid B>_{\dot{H}^s}\le C\|\nabla u\|_{H^s}\|B\|_{H^s}^2,\ \ s>\frac{d}{2},\ d=2,3,
\end{equation*}
Fefferman et al. \cite{Fef2} obtained the local existence and uniqueness for (\ref{1.1}) and related models with the initial data $(u_{0},B_{0})\in H^s(\mathbb{R}^d),\ s>\frac{d}{2}.$ For other results concerning regularity criterions, we refer to \cite{Fan1} and \cite{Zhou1}.

Very recently, Chemin et al. in \cite{1} obtain the local existence for (\ref{1.1}) in 2D and 3D. But for the 2D case, the uniqueness was not obtained. Our main result is filling the gap of their works. The details can be described as follows
\begin{thm}\label{t1.1}
For $u_{0}\in B_{2,1}^0(\mathbb{R}^2)$ and $B_{0}\in B_{2,1}^1(\mathbb{R}^2)$ with ${\rm div}u_{0}={\rm div}B_{0}=0$, there exists a time $T=T(\nu,\|u_{0}\|_{B_{2,1}^0},\|B\|_{B_{2,1}^1})>0$ such that the system (\ref{1.1}) has a unique solution $(u,B)$ with
$$u\in C([0,T];B_{2,1}^0(\mathbb{R}^2))\cap L^1([0,T];B_{2,1}^2)$$
and
$$B\in C([0,T];B_{2,1}^1(\mathbb{R}^2)).$$
\end{thm}

\vskip .3in
\section{Preliminaries}
Let $\mathfrak{B}=\{\xi\in\mathbb{R}^d,\ |\xi|\le\frac{4}{3}\}$ and $\mathfrak{C}=\{\xi\in\mathbb{R}^d,\ \frac{3}{4}\le|\xi|\le\frac{8}{3}\}$. Choose two nonnegative smooth radial function $\chi,\ \varphi$ supported, respectively, in $\mathfrak{B}$ and $\mathfrak{C}$ such that
$$\chi(\xi)+\sum_{j\ge0}\varphi(2^{-j}\xi)=1,\ \ \xi\in\mathbb{R}^d,$$
$$\sum_{j\in\mathbb{Z}}\varphi(2^{-j}\xi)=1,\ \ \xi\in\mathbb{R}^d\setminus\{0\}.$$
We denote $\varphi_{j}=\varphi(2^{-j}\xi),$ $h=\mathfrak{F}^{-1}\varphi$ and $\tilde{h}=\mathfrak{F}^{-1}\chi,$ where $\mathfrak{F}^{-1}$ stands for the inverse Fourier transform. Then the dyadic blocks
$\Delta_{j}$ and $S_{j}$ can be defined as follows
$$\Delta_{j}f=\varphi(2^{-j}D)f=2^{jd}\int_{\mathbb{R}^d}h(2^jy)f(x-y)dy,$$
$$S_{j}f=\sum_{k\le j-1}\Delta_{k}f=\chi(2^{-j}D)f=2^{jd}\int_{\mathbb{R}^d}\tilde{h}(2^jy)f(x-y)dy.$$
Formally, $\Delta_{j}=S_{j}-S_{j-1}$ is a frequency projection to annulus $\{\xi:\ C_{1}2^j\le|\xi|\le C_{2}2^j\}$, and $S_{j}$ is a frequency projection to the ball $\{\xi:\ |\xi|\le C2^j\}$. One can easily verifies that with our choice of $\varphi$
$$\Delta_{j}\Delta_{k}f=0\ if \ |j-k|\ge2\ \ {\rm and}\ \  \Delta_{j}(S_{k-1}f\Delta_{k}f)=0\  if |j-k|\ge5.$$
With the introduction of $\Delta_{j}$ and $S_{j}$, let us recall the definition of the  Besov space.
\par Let $s\in \mathbb{R}$, $(p,q)\in[1,\infty]^2,$ the homogeneous space $\dot{B}_{p,q}^s$ is defined by
$$\dot{B}_{p,q}^{s}=\{f\in \mathfrak{S}';\ \|f\|_{\dot{B}_{p,q}^{s}}<\infty\},$$
where
\begin{equation*}
\|f\|_{\dot{B}_{p,q}^s}=\left\{\begin{aligned}
&\displaystyle (\sum_{j\in \mathbb{Z}}2^{sqj}\|\Delta_{j}f\|_{L^p}^{q})^\frac{1}{q},\ \ \ \ {\rm for} \ \ 1\le q<\infty,\\
&\displaystyle \sup_{j\in\mathbb{Z}}2^{sj}\|\Delta_{j}f\|_{L^p},\ \ \ \ \ \ \ \ \ {\rm for}\ \ q=\infty,\\
\end{aligned}
\right.
\end{equation*}
and $\mathfrak{S}'$ denotes the dual space of $\mathfrak{S}=\{f\in\mathcal{S}(\mathbb{R}^d);\ \partial^{\alpha}\hat{f}(0)=0;\ \forall\ \alpha\in \ \mathbb{N}^d $\ {\rm multi-index}\} and can be identified by the quotient space of $\mathcal{S'}/\mathcal{P}$ with the polynomials space $\mathcal{P}$.
\par Let $s>0,$ and $(p,q)\in [1,\infty]^2$, the inhomogeneous Besov space $B_{p,q}^s$ is defined by
$${B}_{p,q}^{s}=\{f\in \mathcal{S'}(\mathbb {R}^d);\ \|f\|_{{B}_{p,q}^{s}}<\infty\},$$
where
$$\|f\|_{B_{p,q}^s}=\|f\|_{L^p}+\|f\|_{\dot{B}_{p,q}^s}.$$
Let's recall space-time space.
\begin{defn}\label{d1}
Let $s\in\mathbb{R}.$ $1\le p,q,r\le \infty$, $I\subset\mathbb{R}$ is an interval. The homogeneous mixed time-space Besov space $\tilde{L}^r(I;\dot{B}_{p,q}^s)$
is defined as the set of all the distributions $f$ satisfying
$$\|f\|_{\tilde{L}^r(I;\dot{B}_{p,q}^s)}=\left\|2^{sj}\left(\int_{I}\|\Delta_{j}f(\tau)\|_{L^p}^rd\tau\right)^\frac{1}{r}\right\|_{l^q(\mathbb{Z})}<\infty.$$
\end{defn}
For convenience, we sometimes use $\tilde{L}^r_{t}\dot{B}_{p,q}^s$ and $ L^r_{t}\dot{B}_{p,q}^s$  to denote
$\tilde{L}^r(0,t;\dot{B}_{p,q}^s)$ and $L^r(0,t;\dot{B}_{p,q}^s)$.

\vskip .1in
Bernstein's inequalities are useful tools in dealing with Fourier localized functions
and these inequalities trade integrability for derivatives. The following proposition
provides Bernstein type inequalities for fractional derivatives.
\begin{prop}\label{p2.1}
Let $\alpha\ge0$. Let $1\le p\le q\le \infty$.
\begin{enumerate}
\item[1)] If $f$ satisfies
$$
\mbox{supp}\, \widehat{f} \subset \{\xi\in \mathbb{R}^d: \,\, |\xi|
\le K 2^j \},
$$
for some integer $j$ and a constant $K>0$, then
$$
\|(-\Delta)^\alpha f\|_{L^q(\mathbb{R}^d)} \le C_1\, 2^{2\alpha j +
j d(\frac{1}{p}-\frac{1}{q})} \|f\|_{L^p(\mathbb{R}^d)}.
$$
\item[2)] If $f$ satisfies
\begin{equation*}\label{spp}
\mbox{supp}\, \widehat{f} \subset \{\xi\in \mathbb{R}^d: \,\, K_12^j
\le |\xi| \le K_2 2^j \}
\end{equation*}
for some integer $j$ and constants $0<K_1\le K_2$, then
$$
C_1\, 2^{2\alpha j} \|f\|_{L^q(\mathbb{R}^d)} \le \|(-\Delta)^\alpha
f\|_{L^q(\mathbb{R}^d)} \le C_2\, 2^{2\alpha j +
j d(\frac{1}{p}-\frac{1}{q})} \|f\|_{L^p(\mathbb{R}^d)},
$$
where $C_1$ and $C_2$ are constants depending on $\alpha,p$ and $q$
only.
\end{enumerate}
\end{prop}
For more details about Besov space  such as  some useful embedding relations and
the  equivalency
$$\|f\|_{\dot{B}_{2,2}^s}\approx\|f\|_{\dot{H}^s},\qquad \|f\|_{B_{2,2}^s}\approx\|f\|_{H^s},$$
 see \cite{2},\cite{3} and \cite{4}.
\vskip .3in
\section{Proof of The main result}
Before the proof of Theorem \ref{t1.1}, we need the following lemma.
\begin{lem}\label{l3.1}
$\forall\  t>0,$
\begin{equation}\label{3.1}
\int_{0}^{t}\|f(\tau)\|_{\dot{B}_{2,1}^1}d\tau\le C(\|f\|_{\tilde{L}^1_{t}\dot{B}_{2,\infty}^1}+\|f\|_{L^\infty_{t}\dot{B}_{2,\infty}^{-1}})\log\left(e+\frac{\|f\|_{\tilde{L}^1_{t}\dot{B}_{2,\infty}^0}+\|f\|_{\tilde{L}^1_{t}\dot{B}_{2,\infty}^2} }{\|f\|_{\tilde{L}^1_{t}\dot{B}_{2,\infty}^1}+\|f\|_{L^\infty_{t}\dot{B}_{2,\infty}^{-1}}}\right).
\end{equation}
\end{lem}
\begin{proof}
Using the definition of homogeneous Besov space, we have
\begin{equation*}
\begin{aligned}
\|f\|_{L^1_{t}\dot{B}_{2,1}^1}=&\sum_{j\in\mathbb{Z}}2^j\|\Delta_{j}f\|_{L^1_{t}L^2}\\
=&\sum_{j<-N}2^j\|\Delta_{j}f\|_{L^1_{t}L^2}
+\sum_{-N\le j\le N}2^j\|\Delta_{j}f\|_{L^1_{t}L^2}+\sum_{j>N}2^j\|\Delta_{j}f\|_{L^1_{t}L^2}\\
\le& 2^{-N}\|f\|_{\tilde{L}^1_{t}\dot{B}_{2,\infty}^0}+2N\|f\|_{\tilde{L}^1_{t}\dot{B}_{2,\infty}^1}+2^{-N}\|f\|_{\tilde{L}^1_{t}\dot{B}_{2,\infty}^2}.
\end{aligned}
\end{equation*}
Choosing
$$N=\log\left(e+\frac{\|f\|_{\tilde{L}^1_{t}\dot{B}_{2,\infty}^0}+\|f\|_{\tilde{L}^1_{t}\dot{B}_{2,\infty}^2}                             }{\|f\|_{\tilde{L}^1\dot{B}_{2,\infty}^1}+\|f\|_{L^\infty_{t}\dot{B}_{2,\infty}^{-1}}}\right),$$
we can get the inequality (\ref{3.1})
\end{proof}
Now, we begin the proof of Theorem \ref{t1.1}. The existence of the solution to (\ref{1.1}) was obtained in \cite{1}, while the continuity in time
can be obtained by the definition of Besov space. So here we only deal with the uniqueness.
Let $(u_{j},B_{j}),\ j=1,2,$ be two solution of (\ref{1.1}), denote $\delta u=u_{1}-u_{2}$, $\delta B=B_{1}-B_{2}$ and $\delta p=p_{1}-p_{2}$, then we obtain
\begin{equation}\label{3.2}
\partial_{t}\delta u+(u_{1}\cdot\nabla)\delta u+(\delta u\cdot\nabla)u_{2}-\nu\Delta\delta u+\nabla\delta p=(B_{1}\cdot\nabla)\delta B+(\delta B\cdot\nabla)B_{2}
\end{equation}
and
\begin{equation}\label{3.3}
\partial_{t}\delta B+(u_{1}\cdot\nabla)\delta B+(\delta u\cdot\nabla)B_{2}=(B_{1}\cdot\nabla)\delta u+(\delta B\cdot\nabla)u_{2}.
\end{equation}
First, we consider (\ref{3.2}). By a standard argument, we have
\begin{equation*}
\begin{aligned}
\frac{d}{dt}\|\Delta_{j}\delta u\|_{L^2}&+\nu2^{2j}\|\Delta_{j}\delta u\|_{L^2}\le \|[\Delta_{j},u_{1}\cdot\nabla]\delta u\|_{L^2}\\
&+\|\Delta_{j}(\delta u\cdot\nabla u_{2})\|_{L^2}+\|\Delta_{j}(B_{1}\cdot\nabla\delta B)\|_{L^2}+\|\Delta_{j}(\delta B\cdot\nabla B_{2})\|_{L^2},
\end{aligned}
\end{equation*}
which with Gronwall's inequality yields that
\begin{equation*}
\begin{aligned}
\|\Delta_{j}\delta u\|_{L^2}\le& \int_{0}^{t} e^{\nu2^{2j}(\tau-t)}(\|[\Delta_{j},u_{1}\cdot\nabla]\delta u\|_{L^2}\\
&+\|\Delta_{j}(\delta u\cdot\nabla u_{2})\|_{L^2}+\|\Delta_{j}(B_{1}\cdot\nabla\delta B)\|_{L^2}+\|\Delta_{j}(\delta B\cdot\nabla B_{2})\|_{L^2})d\tau.
\end{aligned}
\end{equation*}
Taking the $L^r(0,t)$ norm, and using Young's inequality to obtain
\begin{equation*}
\begin{aligned}
\|\Delta_{j}\delta u\|_{L^r_{t}L^2}\le& \|e^{-\nu2^{2j}\tau}\|_{L^r_{t}} (\|[\Delta_{j},u_{1}\cdot\nabla]\delta u\|_{L^1_{t}L^2}\\
&+\|\Delta_{j}(\delta u\cdot\nabla u_{2})\|_{L^1_{t}L^2}+\|\Delta_{j}(B_{1}\cdot\nabla\delta B)\|_{L^1_{t}L^2}+\|\Delta_{j}(\delta B\cdot\nabla B_{2})\|_{L^1_{t}L^2})\\
\le&(\nu2^{2j}r)^{-\frac{1}{r}}(\|[\Delta_{j},u_{1}\cdot\nabla]\delta u\|_{L^1_{t}L^2}\\
&+\|\Delta_{j}(\delta u\cdot\nabla u_{2})\|_{L^1_{t}L^2}+\|\Delta_{j}(B_{1}\cdot\nabla\delta B)\|_{L^1_{t}L^2}+\|\Delta_{j}(\delta B\cdot\nabla B_{2})\|_{L^1_{t}L^2}).
\end{aligned}
\end{equation*}
Multiplying $2^{-j}$, and taking the $l^\infty$ norm, we obtain
\begin{equation*}
\begin{aligned}
\nu^\frac{1}{r}\|\delta u\|_{\tilde{L}^r_{t}\dot{B}_{2,\infty}^{-1+\frac{2}{r}}}\le& \sup_{j\in\mathbb{Z}}2^{-j}(\|[\Delta_{j},u_{1}\cdot\nabla]\delta u\|_{L^1_{t}L^2}\\
&+\|\Delta_{j}(\delta u\cdot\nabla u_{2})\|_{L^1_{t}L^2}+\|\Delta_{j}(B_{1}\cdot\nabla\delta B)\|_{L^1_{t}L^2}+\|\Delta_{j}(\delta B\cdot\nabla B_{2})\|_{L^1_{t}L^2})\\
=&K_{1}+K_{2}+K_{3}+K_{4},
\end{aligned}
\end{equation*}
where
\begin{align*}
&K_1 =\sup_{j\in\mathbb{Z}}2^{-j}\|[\Delta_{j},u_{1}\cdot\nabla]\delta u\|_{L^1_{t}L^2}, \qquad K_2 = \sup_{j\in\mathbb{Z}}2^{-j}\|\Delta_{j}(\delta u\cdot\nabla u_{2})\|_{L^1_{t}L^2},\\
&K_3 = \sup_{j\in\mathbb{Z}}2^{-j}\|\Delta_{j}(B_{1}\cdot\nabla\delta B)\|_{L^1_{t}L^2}, \qquad K_4 =\sup_{j\in\mathbb{Z}}2^{-j}\|\Delta_{j}(\delta B\cdot\nabla B_{2})\|_{L^1_{t}L^2}.
\end{align*}
In the following, we will bound $K_{i}$, $i=1,2,3,4.$ By homogeneous Bony decomposition, we can split $K_{1}$ into four parts,
\begin{equation}\label{3.4}
\begin{aligned}
K_{1}\le& \sup_{j\in\mathbb{Z}}2^{-j}\sum_{|k-j|\le4}\|[\Delta_{j},S_{k-1}u_{1}\cdot\nabla]\Delta_{k}\delta u\|_{L^1_t L^2}
+\sup_{j\in\mathbb{Z}}2^{-j}\sum_{|k-j|\le4}\|\Delta_{j}(\Delta_{k}u_{1}\cdot\nabla S_{k-1}\delta u)\|_{L^1_t L^2}\\
&+\sup_{j\in\mathbb{Z}}2^{-j}\sum_{k\ge j-3}\|\Delta_{k}u_{1}\cdot\nabla\Delta_{j}S_{k+1}\delta u\|_{L^1_t L^2}
+\sup_{j\in\mathbb{Z}}2^{-j}\sum_{k\ge j-3}\|\Delta_{j}(\Delta_{k}u_1 \cdot\nabla\tilde{\Delta}_k\delta u)\|_{L^1_t L^2}\\
=&K_{11}+K_{12}+K_{13}+K_{14},
\end{aligned}
\end{equation}
where $\tilde{\Delta}_{k}=\Delta_{k-1}+\Delta_{k}+\Delta_{k+1}.$\\
By H\"{o}lder's inequality, standard commutator estimate and Bernstein's inequality,
$$K_{11}\le C\sup_{j\in\mathbb{Z}}2^{-j}\|\nabla u_{1}\|_{L^{1}_{t}L^\infty}\|\Delta_{j}\delta u\|_{L^\infty_{t}L^2}\le C\|u_{1}\|_{L^1_{t}\dot{B}_{2,1}^2}\|\delta u\|_{L^\infty_{t}\dot{B}_{2,\infty}^{-1}},$$
$$K_{12}\le C\sup_{j\in\mathbb{Z}}2^{-j}\|\Delta_{j} u_{1}\|_{L^{1}_{t}L^\infty}\|\nabla S_{j-1}\delta u\|_{L^\infty_{t}L^2}\le C\|u_{1}\|_{L^1_{t}\dot{B}_{2,1}^2}\|\delta u\|_{L^\infty_{t}\dot{B}_{2,\infty}^{-1}},$$
\begin{align*}
K_{13}\le& C\sup_{j\in\mathbb{Z}}2^{-j}\sum_{k\ge j-3}2^j\|\Delta_{j}\delta u\|_{L^\infty_{t}L^2}\|\Delta_{k}u_{1}\|_{L^1_{t}L^\infty}\\
\le& C\sup_{j\in\mathbb{Z}}\sum_{k\ge j-3}2^{j-k}2^{-j}\|\Delta_{j}\delta u\|_{L^\infty_{t}L^2}2^k\|\Delta_{k}u_{1}\|_{L^1_{t}L^\infty}\\
\le& C\|u_{1}\|_{L^1_{t}\dot{B}_{2,1}^2}\|\delta u\|_{L^\infty_{t}\dot{B}_{2,\infty}^{-1}}
\end{align*}
and
$$K_{14}\le C\sup_{j\in\mathbb{Z}}\sum_{k\ge j-3}2^j\|\Delta_{k}u_{1}\|_{L^1_{t}L^2}\|\tilde{\Delta}_{k}\delta u\|_{L^\infty_{t}L^2}\le C\|u_{1}\|_{L^1_{t}\dot{B}_{2,1}^2}\|\delta u\|_{L^\infty_{t}\dot{B}_{2,\infty}^{-1}}.$$
Collecting the estimates above in (\ref{3.4}), we obtain
$$K_{1}\le C\|u_{1}\|_{L^1_{t}\dot{B}_{2,1}^2}\|\delta u\|_{L^\infty_{t}\dot{B}_{2,\infty}^{-1}}.$$
By homogeneous Bony decomposition again,
\begin{align*}
K_{2}\le& \sup_{j\in\mathbb{Z}}2^{-j}\sum_{|k-j|\le4}\|\Delta_{j}(\Delta_{k}\delta u\cdot\nabla S_{k-1}u_{2})\|_{L^1_{t}L^2}\\
&+\sup_{j\in\mathbb{Z}}2^{-j}\sum_{|k-j|\le4}\|\Delta_{j}(S_{k-1}\delta u\cdot\nabla \Delta_{k}u_{2})\|_{L^1_{t}L^2}\\
&+\sup_{j\in\mathbb{Z}}2^{-j}\sum_{k\ge j-3}\|\Delta_{j}(\Delta_{k}\delta u\cdot\nabla \tilde{\Delta}_{k}u_{2})\|_{L^1_{t}L^2}\\
=&K_{21}+K_{22}+K_{23}.
\end{align*}
By H\"{o}lder's inequality and Bernstein's inequality,
$$K_{21}\le C\|\nabla u_{2}\|_{L^1_{t}L^\infty}\sup_{j\in\mathbb{Z}}2^{-j}\|\Delta_{j}\delta u\|_{L^\infty_{t}L^2}\le C\|u_{2}\|_{L^1_{t}\dot{B}_{2,1}^2}\|\delta u\|_{L^\infty_{t}\dot{B}_{2,\infty}^{-1}},$$
\begin{align*}
K_{22}\le& C\sup_{j\in\mathbb{Z}}2^{-j}\|S_{j-1}\delta u\|_{L^\infty_{t}L^\infty}\|\nabla \Delta_{j}u_{2}\|_{L^1_{t}L^2}\\
\le& C\sup_{j\in\mathbb{Z}}2^j\|\nabla\Delta_{j}u\|_{L^1_{t}L^2}2^{-j}\|S_{j-1}\delta u\|_{L^\infty_{t}L^2}\\
\le& C\|u_{2}\|_{L^1_{t}\dot{B}_{2,1}^2}\|\delta u\|_{L^\infty_{t}\dot{B}_{2,\infty}^{-1}},
\end{align*}
\begin{align*}
K_{23}\le& C\sup_{j\in\mathbb{Z}}\sum_{k\ge j-3}2^j\|\Delta_{k}\delta u\|_{L^\infty_{t}L^2}\|\tilde{\Delta}_{k} u_{2}\|_{L^1_{t}L^2}\\
\le& C\sup_{j\in\mathbb{Z}}\sum_{k\ge j-3}2^{j-k}2^{2k}\|\tilde{\Delta}_{k}u_{2}\|_{L^1_{t}L^2}2^{-k}\|\Delta_{k}\delta u\|_{L^\infty_{t}L^2}\\
\le& C\|u_{2}\|_{L^1_{t}\dot{B}_{2,1}^2}\|\delta u\|_{L^\infty_{t}\dot{B}_{2,\infty}^{-1}}.
\end{align*}
Thus we have
$$K_{2}\le C\|u_{2}\|_{L^1_{t}\dot{B}_{2,1}^2}\|\delta u\|_{L^\infty_{t}\dot{B}_{2,\infty}^{-1}}.$$
Similarly,  we can bound $K_{3}$ and $K_{4}$ as follows:
$$K_{3}\le \|B_{1}\cdot\nabla \delta B\|_{\tilde{L}^1_{t}\dot{B}_{2,\infty}^{-1}}\le \int_{0}^{t}\|B_{1}\cdot\nabla\delta B\|_{\dot{B}_{2,\infty}^{-1}}d\tau\le
\int_{0}^{t}\|B_{1}\|_{\dot{B}_{2,1}^1}\|\delta B\|_{\dot{B}_{2,\infty}^0}d\tau
$$
and
$$K_{4}\le \|\delta B\cdot\nabla B_{2}\|_{\tilde{L}^1_{t}\dot{B}_{2,\infty}^{-1}}\le \int_{0}^{t}\|\delta B\cdot\nabla B_{2}\|_{\dot{B}_{2,\infty}^{-1}}d\tau\le
\int_{0}^{t}\|B_{2}\|_{\dot{B}_{2,1}^1}\|\delta B\|_{\dot{B}_{2,\infty}^0}d\tau.
$$
Therefore,
\begin{equation}\label{3.5}
\begin{aligned}
\|\delta u\|_{L^\infty_{t}\dot{B}_{2,\infty}^{-1}}&+\nu\|\delta u\|_{\tilde{L}^1_{t}\dot{B}_{2,\infty}^1}\\
\le& C\|(u_{1},u_{2})\|_{L^1_{t}\dot{B}_{2,1}^2}\|\delta u\|_{L^\infty_{t}\dot{B}_{2,\infty}^{-1}}+C\int_{0}^{t}\|(B_{1},B_{2})\|_{\dot{B}_{2,1}^1}\|\delta B\|_{\dot{B}_{2,\infty}^0}d\tau.
\end{aligned}
\end{equation}
Next, we consider (\ref{3.3}), we have the following estimate,
\begin{equation*}
\begin{aligned}
\frac{d}{dt}\|\delta B\|_{\dot{B}_{2,\infty}^0}\le& \sup_{j\in\mathbb{Z}}\|[\Delta_{j},u_{1}\cdot\nabla]\delta B\|_{L^2}\\
&+\sup_{j\in\mathbb{Z}}\|\Delta_{j}(\delta u\cdot\nabla B_{2})\|_{L^2}+\sup_{j\in\mathbb{Z}}\|\Delta_{j}(B_{1}\cdot\nabla \delta u)\|_{L^2}+\sup_{j\in\mathbb{Z}}\|\Delta_{j}(\delta B\cdot\nabla u_{2})\|_{L^2}\\
=&J_{1}+J_{2}+J_{3}+J_{4}.
\end{aligned}
\end{equation*}
By homogeneous Bony decomposition,
\begin{equation*}
\begin{aligned}
J_{1}\le& \sup_{j\in\mathbb{Z}}\sum_{|k-j|\le 4}\|[\Delta_{j},S_{k-1}u_{1}\cdot\nabla]\Delta_{k}\delta B\|_{L^2}+\sup_{j\in\mathbb{Z}}\sum_{|k-j|\le4}\|\Delta_{j}(\Delta_{k}u_{1}\cdot\nabla S_{k-1}\delta B)\|_{L^2}\\
&+\sup_{j\in\mathbb{Z}}\sum_{k\ge j-3}\|\Delta_{k}u_{1}\cdot\nabla S_{k+1}\Delta_{j}\delta B\|_{L^2}+
\sup_{j\in\mathbb{Z}}\sum_{k\ge j-3}\|\Delta_{j}(\Delta_{k}u_{1}\cdot\nabla\tilde{\Delta}_{k}\delta B)\|_{L^2}\\
=&J_{11}+J_{12}+J_{13}+J_{14}.
\end{aligned}
\end{equation*}
By H\"{o}lder's inequality and Bernstein's inequality,
$$J_{11}\le C\sup_{j\in\mathbb{Z}}\|\nabla u_{1}\|_{L^\infty}\|\Delta_{j}\delta B\|_{L^2}\le C\|u_{1}\|_{\dot{B}_{2,1}^2}\|\delta B\|_{\dot{B}_{2,\infty}^0},$$
$$J_{12}\le C\sup_{j\in\mathbb{Z}}2^{2j}\|\Delta_{j}u_{1}\|_{L^2}2^{-2j}\|\nabla S_{j-1}\delta B\|_{L^\infty}\le C\|u_{1}\|_{\dot{B}_{2,1}^2}\|\delta B\|_{\dot{B}_{2,\infty}^0},$$
\begin{align*}
J_{13}\le& C\sup_{j\in\mathbb{Z}}\sum_{k\ge j-3}2^j\|\Delta_{k}u\|_{L^\infty}\|\Delta_{j}\delta B\|_{L^2}\\
\le& C\sup_{j\in\mathbb{Z}}\sum_{k\ge j-3}2^{j-k}2^{2k}\|\Delta_{k}u\|_{L^2}\|\Delta_{j}\delta B\|_{L^2}
\le C\|u_{1}\|_{\dot{B}_{2,1}^2}\|\delta B\|_{\dot{B}_{2,\infty}^0},
\end{align*}
$$J_{14}\le C\sup_{j\in\mathbb{Z}}\sum_{k\ge j-3}2^{2j}\|\Delta_{k}u\|_{L^2}\|\tilde{\Delta}_{k}\delta B\|_{L^2}\le C\|u_{1}\|_{\dot{B}_{2,1}^2}\|\delta B\|_{\dot{B}_{2,\infty}^0}.$$
Hence we have
$$J_{1}\le C\|u_{1}\|_{\dot{B}_{2,1}^2}\|\delta B\|_{\dot{B}_{2,\infty}^0}.$$
By the inequality
\begin{equation*}
\|fg\|_{\dot{B}_{2,1}^{1}}\le C\|f\|_{\dot{B}_{2,1}^{1}}\|g\|_{\dot{B}_{2,1}^{1}},
\end{equation*}
(see, e.g., \cite{5}), we have
$$J_{2}+J_{3}\le \|\delta u\cdot \nabla B_{2}\|_{\dot{B}_{2,\infty}^0}+\|B_{1}\cdot \nabla\delta u\|_{\dot{B}_{2,\infty}^0}\le C\|\delta u\|_{\dot{B}_{2,1}^1}\|(B_{1},B_{2})\|_{\dot{B}_{2,1}^1}.$$
Finally, we bound $J_{4}$. By homogeneous Bony decomposition,
\begin{align*}
J_{4}\le& \sup_{j\in \mathbb{Z}}\sum_{|k-j|\le4}\|\Delta_{j}(\Delta_{k}\delta B\cdot\nabla S_{k-1}u_{2})\|_{L^2}+
\sup_{j\in \mathbb{Z}}\sum_{|k-j|\le4}\|\Delta_{j}(S_{k-1}\delta B\cdot\nabla \Delta_{k}u_{2})\|_{L^2}\\
&+\sup_{j\in \mathbb{Z}}\sum_{k\ge j-3}\|\Delta_{j}(\Delta_{k}\delta B\cdot\nabla\tilde{\Delta}_{k}u_{2})\|_{L^2}\\
=&J_{41}+J_{42}+J_{43},
\end{align*}
where
$$J_{41}\le C\sup_{j\in\mathbb{Z}}\|\nabla u_{2}\|_{L^\infty}\|\Delta_{j}\delta B\|_{L^2}\le C\|u_{2}\|_{\dot{B}_{2,1}^2}\|\delta B\|_{\dot{B}_{2,\infty}^0},$$
$$J_{42}\le C\sup_{j\in\mathbb{Z}}2^j\|\nabla\Delta_{j}u_{2}\|_{L^2}2^{-j}\|S_{j-1}\delta B\|_{L^\infty}\le C\|u_{2}\|_{\dot{B}_{2,1}^2}\|\delta B\|_{\dot{B}_{2,\infty}^0}$$
and
$$J_{43}\le C\sup_{j\in\mathbb{Z}}\sum_{k\ge j-3}2^{2j}\|\tilde{\Delta}_{k}u_{2}\|_{L^2}\|\Delta_{k}\delta B\|_{L^2}\le C\|u_{2}\|_{\dot{B}_{2,1}^2}\|\delta B\|_{\dot{B}_{2,\infty}^0}.$$
So
$$J_{4}\le C\|u_{2}\|_{\dot{B}_{2,1}^2}\|\delta B\|_{\dot{B}_{2,\infty}^0}.$$
Therefore,
$$
\frac{d}{dt}\|\delta B\|_{\dot{B}_{2,\infty}^0}\le C\|(u_{1},u_{2})\|_{\dot{B}_{2,1}^2}\|\delta B\|_{\dot{B}_{2,\infty}^0}+C\|(B_{1},B_{2})\|_{\dot{B}_{2,1}^1}\|\delta u\|_{\dot{B}_{2,1}^1},
$$
which implies that $\forall\ 0\le t\le T,$ $T$ is the lifespan of the solution,
\begin{equation}\label{3.6}
\|\delta B\|_{L^\infty_{t}\dot{B}_{2,\infty}^0}\le C\|(u_{1},u_{2})\|_{L^1_{t}\dot{B}_{2,1}^2}\|\delta B\|_{L^\infty_{t}\dot{B}_{2,\infty}^0}+
C\|(B_{1},B_{2})\|_{L^\infty_{t}\dot{B}_{2,1}^1}\|\delta u\|_{L^1_{t}\dot{B}_{2,1}^1}.
\end{equation}
Set $0<\bar{T}<T$ such that
$$\int_{0}^{\bar{T}}\|(u_{1},u_{2})\|_{\dot{B}_{2,1}^2}dt<\frac{1}{4C},$$
then  $\forall 0\le t\le \bar{T}$, (\ref{3.5}) and (\ref{3.6}) reduce to
\begin{equation}\label{3.7}
\frac{3}{4}\|\delta u\|_{L^\infty_{t}\dot{B}_{2,\infty}^{-1}}+\nu\|\delta u\|_{\tilde{L}^1_{t}\dot{B}_{2,\infty}^1}\le
\int_{0}^{t}\|(B_{1},B_{2})\|_{\dot{B}_{2,1}^1}\|\delta B\|_{\dot{B}_{2,\infty}^0}d\tau
\end{equation}
and
\begin{equation}\label{3.8}
\frac{3}{4}\|\delta B\|_{L^\infty_{t}\dot{B}_{2,\infty}^{0}}\le C\|(B_{1},B_{2})\|_{L^\infty_{t}\dot{B}_{2,1}^1}\|\delta u\|_{L^1_{t}\dot{B}_{2,1}^1}.
\end{equation}
Plugging (\ref{3.8}) into (\ref{3.7}) yields that
\begin{align*}
\frac{3}{4}\|\delta u\|_{L^\infty_{t}\dot{B}_{2,\infty}^{-1}}+\nu\|\delta u\|_{\tilde{L}^1_{t}\dot{B}_{2,\infty}^1}\le&
C\int_{0}^{t}\|(B_{1},B_{2})(\tau)\|_{\dot{B}_{2,1}^1}\|(B_{1},B_{2})\|_{L^\infty_{\tau}\dot{B}_{2,1}^1}\|\delta u\|_{L^1_{\tau}\dot{B}_{2,1}^1}d\tau\\
\le&C_{T}\int_{0}^{t}\|\delta u\|_{L^1_{\tau}\dot{B}_{2,1}^1}d\tau.
\end{align*}
Thanks to the Log-type inequality (\ref{3.1}), denote
$$X(t)=\|\delta u\|_{L^\infty_{t}\dot{B}_{2,\infty}^{-1}}+\|\delta u\|_{\tilde{L}^1_{t}\dot{B}_{2,\infty}^1},$$ we have
\begin{align*}
X(t)\le
C_{T,\nu}\int_{0}^{t} X(\tau)\log\left(e+\frac{V(\tau)}{X(\tau)}\right)d\tau,
\end{align*}
where $V(t)=\|\delta u\|_{\tilde{L}^1_{t}\dot{B}_{2,\infty}^0}+\|\delta u\|_{\tilde{L}^1_{t}\dot{B}_{2,\infty}^2}$, is bounded in $[0,T]$. Applying the Osgood's Lemma (see,\cite{3} p.125) and combining with (\ref{3.6}) can yields $\delta u=\delta B=0$ in $[0,\bar{T}].$ By a standard continuous argument, we can show that $\delta u=\delta B=0$ in $[0,T]\times \mathbb{R}^2.$ This completes the proof of Theorem \ref{1.1}.

\vskip .4in
\section*{Acknowledgements}
The author would like to thank   Prof. J-Y. Chemin  for his
comments.

\end{document}